\documentclass[11pt, reqno]{amsart}
\usepackage{mathrsfs}
\usepackage[active]{srcltx}
\usepackage{mathrsfs,amsmath}
\usepackage{mathtools}
\usepackage{longtable}
\usepackage{todonotes}
\usepackage{amssymb}
\usepackage{cite}
\allowdisplaybreaks

\usepackage[dvipsnames]{xcolor}
\definecolor{bluecite}{HTML}{0875b7}

\usepackage[unicode=true,
bookmarksopen={true},
pdffitwindow=true,
colorlinks=true,
linkcolor=bluecite,
citecolor=bluecite,
urlcolor=bluecite,
hyperfootnotes=false,
pdfstartview={FitH},
pdfpagemode= UseNone]{hyperref}

\DeclareMathOperator*{\argmin}{arg\,min}

\newcommand{\ler}[1]{\left(#1\right)}
\newcommand{\lers}[1]{\left\{#1\right\}}
\newcommand{\dd}{\mathrm{d}}

\newcommand{\be}{\begin{equation}}
\newcommand{\ee}{\end{equation}}

\newcommand{\abs}[1]{\left|#1\right|}
\newcommand{\norm}[1]{\left|\left|#1\right|\right|}

\newcommand{\Rn}{\mathbb{R}^n}

\newcommand{\supp}{\mathrm{supp}}

\newcommand{\isom}{\mathrm{Isom}}

\newcommand{\R}{\mathbb{R}}

\usepackage[left=3cm,right=3cm,top=3.0cm,bottom=3.0cm]{geometry}
\usepackage{marginnote}
\usepackage[normalem]{ulem}

\newtheorem{theorem}{Theorem}[section]

\newtheorem{lemma}[theorem]{Lemma}
\newtheorem{corollary}[theorem]{Corollary}

\newtheorem{remark}[theorem]{Remark}
\numberwithin{equation}{section}

\usepackage{bbm}  
\usepackage{dsfont}


\title{Restoring Wasserstein Rigidity with a single point}

\author[Zolt\'an M. Balogh]{Zolt\'an M. Balogh}
\address{Zolt\'an M. Balogh, Universit\"at Bern\\ Mathematisches Institut (MAI)\\ Sidlerstrasse 12\\ 3012 Bern\\ Schweiz}
\email{zoltan.balogh@unibe.ch}

\author[Eric Str\"oher]{Eric Str\"oher}
\address{Eric Str\"oher, Universit\"at Bern\\ Mathematisches Institut (MAI)\\ Sidlerstrasse 12\\ 3012 Bern\\ Schweiz}
\email{eric.stroeher@unibe.ch}

\author[D\'aniel Virosztek]{D\'aniel Virosztek}
\address{D\'aniel Virosztek, HUN-REN Alfr\'ed R\'enyi Institute of Mathematics\\ Re\'altanoda u. 13-15.\\Budapest H-1053\\ Hungary}
\email{virosztek.daniel@renyi.hu}

\thanks{Z. M. Balogh and E. Str\"oher are supported by the Swiss National Science Foundation, Grant Nr. {200020\_228012}.}

\thanks{D. Virosztek is supported by the Momentum program of the Hungarian Academy of Sciences under grant agreement no. LP2021-15/2021, by the Hungarian National Research, Development and Innovation Office (NKFIH) under grant agreement no. Excellence\_151232, and partially supported by the ERC Synergy Grant No. 810115.}

\begin{document}
    \begin{abstract}
    We consider isometrically flexible Wasserstein spaces and demonstrate that adding a single point to the underlying metric space makes these Wasserstein spaces rigid. 
    \end{abstract}

\maketitle 

\tableofcontents

%%%%%%%%%%%%%%%%%%%%%%%%%%%%%%%%%%%%%%%%%%%%%%%%%%%%%%%%%%%%%%%%%%%%%%
%%%%%%%%%%%%%%%%%%%%%%%%%%%%%%%%%%%%%%%%%%%%%%%%%%%%%%%%%%%%%%%%%%%%%%
\section{Introduction}

Given a Polish metric space $(X, d)$ and a parameter $p \geq 1$, the $p$-Wasserstein space $\mathcal{W}_p(X, d)$ is the set of probability measures $\mu$ with finite $p$-moment. To be more precise, a Borel probability measure $\mu$ is in the $p$-Wasserstein space if
$$
\int_X d^p(x_0, x)d\mu(x)< \infty
$$ 
for some (and hence any) $x_0 \in X$. We can consider the Wasserstein distance between two measures
$$
d_{\mathcal{W}_p}(\mu, \nu)= \inf_{\pi \in \Pi(\mu, \nu)} \left(\int_{X \times X} d^p(x, y)d \pi(x, y)\right)^{\frac{1}{p}},
$$ 
where $\Pi(\mu, \nu)$ is the set of couplings between $\mu$ and $\nu$. We recall that $\pi \in \Pi(\mu, \nu)$ if $\pi \in \mathcal{P}(X \times X)$ with $p_{x\#}\pi=\mu$ and $p_{y\#}\pi=\nu$ for the functions $p_x(x, y)=x$ and $p_y(x,y)=y$, and $\#$ represents the push-forward map presented in \eqref{eq:pushforward}. Many properties of Wasserstein spaces have been already extensively studied, for an overview of the area we refer to the monographs \cite{ambrosio_lectures_2021}, \cite{figalli_invitation_2021} or \cite{villani_optimal_2009}.

%%%%%%%%%%%%%%%%%%%%%%%%%%%%%%%%%%%%%%%%%%%%%%%%%%%%%%%%%%%%%%%%%%%%%%
\subsection{Motivation and background results}
One question that has only been asked in recent years about Wasserstein spaces concerns their isometry groups. Indeed, if $\phi$ is a measurable map from a metric space $(X, d)$ to itself, we can consider the push-forward map $\phi_{\#}$ defined as 
\begin{equation}
\label{eq:pushforward}
\phi_{\#}(\mu)(A)= \mu (\phi^{-1}(A)),
\end{equation}
for any $A \subset X$. This push-forward map is a map acting on probability measures, and if $\phi$ is an isometry on $(X, d)$, it is easy to see that $\phi_{\#}$ is also an isometry from the Wasserstein space $\mathcal{W}_p(X, d)$ to itself. We recall that an isometry is a distance-preserving surjective map.

 An isometry of the Wasserstein space $\Phi$ that can be written as the the push-forward of an isometry of $(X, d)$, i.e. $\Phi=\phi_{\#}$, is called a trivial isometry, and the push-forward operation induces an embedding between the isometry group of $(X,d)$ and the isometry group of $\mathcal{W}_p(X, d)$. 

A central problem in this context is whether the isometry group of the Wasserstein space is strictly bigger than the isometry group of the base space. If this is the case, we say that the Wasserstein space  $\mathcal{W}_p(X, d)$ is flexible; if all isometries are trivial, then $\mathcal{W}_p(X, d)$ is isometrically rigid. 

The first flexibility result was given by Kloeckner \cite{kloeckner_geometric_2010}, who showed that for $p=2$, the Wasserstein space $\mathcal{W}_2(\Rn, d_E)$ is flexible for any $n \geq 1$, where $d_E$ is the standard Euclidean distance. This result was extended by \cite{geher_isometry_2022}, where the authors showed that $\mathcal{W}_2(H, d_H)$ allows non-trivial isometries whenever $(H, d_H)$ is a separable Hilbert space; as shown in \cite{che_isometric_2024}, this is still true if the base space is given by the product of any proper metric space and a separable Hilbert space, $X=Y \times H$ (with the $l_2$ product metric). 
Another way of constructing flexible $p$-Wasserstein spaces for all $p \geq 1$ is to consider the interval $[0,1]$ as base space with the $p$-dependent metric $d(x,y)=|x-y|^{1/p}$, as shown in \cite{geher_isometric_2020} and \cite{geher_isometry_2022}. 

An interesting aspect is that these Wasserstein spaces are (so far) the only examples of flexible spaces that have been found. For all other spaces which have been investigated under this premise it has been shown that the Wasserstein space only has trivial isometries. For example, Bertrand and Klockner showed in \cite{bertrand_geometric_2012} that if $(X,d)$ is a Riemannian manifold with negative sectional curvature (also known as a Hadamard space), then $\mathcal{W}_2(X, d)$ is isometrically rigid. The same is true if the manifold is closed, with strictly positive curvature, as shown by Santos-Rodrigues in \cite{santos-rodriguez_isometries_2022}. In \cite{balogh_wasserstein_2025}, Balogh, Str\"oher, Titkos and Virosztek showed that when the norm $N$ is sufficiently smooth, the Wasserstein space $\mathcal{W}_p(\Rn, d_N)$ is isometrically rigid for $p \neq 2$. 
Even in the subriemannian case Balogh, Titkos and Virosztek showed in \cite{balogh_isometries_2025} and \cite{balogh_isometric_2026} the rigidity of the Wasserstein space over the Heisenberg group equipped with the Koranyi metric for any $p \geq 1$, and for a general Carnot group when $p=1$. 

Furthermore, even for flexible Wasserstein spaces, their flexibility seems to be instable, in the sense that a small variation of a parameter of the space is enough to create a rigid Wasserstein space. The easiest way to see this is to look at the parameter $p$. If we consider as the base space $(X, d)$ the Euclidean plane (resp. the $[0,1]$ interval with the absolute value distance), then as described above for the parameter $p=2$ (resp. $p=1$) the Wasserstein space $\mathcal{W}_p(X, d)$ is isometrically flexible. However, for any other value of the parameter $p$, even very close to $2$ (resp. close to $1$), the results of \cite{geher_isometry_2022} and \cite{geher_isometric_2020} show that the space $\mathcal{W}_p(X, d)$ is once again isometrically rigid. 

This instability is also visible when we consider the metric of the underlying space. Indeed, we can consider the Euclidean norm to be the $l_2$ norm, where the $l_q$ norm is given by the formula 
$$
l_q(x)=\left(\sum_{i=1}^n |x_i|^q \right)^{1/q}.
$$ Then, even though for $q=2$ the Wasserstein space $\mathcal{W}_2(\Rn, l_2)$ is isometrically flexible, \cite{balogh_wasserstein_2025} and \cite{stroeher_lp_case} showed that for any $q \neq 2, q \geq 1$ the space $\mathcal{W}_p(\Rn, l_q)$ is rigid for all $p \geq 1$; once again a slight perturbation of a parameter is enough to restore the rigidity of the Wasserstein space. 

\subsection{Main theorems}
In this article, we show that by starting from a flexible Wasserstein space, the rigidity property is restored if we slightly modify the underlying space. Indeed, we will show in both the case of the interval $[0,1]$ and the case of the Euclidean plane $\R^2$ that adding a single point, which we will denote by $q$, is enough to break the flexibility of the Wasserstein space. Notice that, while we only consider the case of the plane, the same argument shows that adding a single point to any (finite-dimensional) Hilbert space equipped with a similar metric will give the same rigidity in its Wasserstein space. 

For the set $[0,1] \cup \{q\}$, we will consider a distance $d_1$ which places the point $q$ uniformly far away from the interval $[0,1]$; we thus have the following Theorem.

\begin{theorem}
\label{thm_1}
The $1$-Wasserstein space $\mathcal{W}_1([0,1]\cup\{q\}, d_1)$ is isometrically rigid. That is, for all $\Phi \in \isom\ler{\mathcal{W}_1([0,1]\cup\{q\}, d_1)}$ there is a $\psi \in \isom\ler{[0,1]\cup\{q\}, d_1}$ such that $\Phi=\psi_{\#}.$
\end{theorem}

For $\R^2 \cup \{q\}$, we embed the plane horizontally into $\R^3$, set $q=(0,0,1)$, and we consider the distance $d_2$ on $\R^2 \cup \{q\}$ to be the standard Euclidean distance of $\R^3$. Then we have the following Theorem.

\begin{theorem}
\label{thm_2}
The $2$-Wasserstein space $\mathcal{W}_2(\R^2\cup\{q\}, d_2)$ is isometrically rigid. That is, for all $\Phi \in \isom\ler{\mathcal{W}_2(\R^2\cup\{q\}, d_2)}$ there is a $\psi \in \isom\ler{\R^2\cup\{q\}, d_2}$ such that $\Phi=\psi_{\#}.$
\end{theorem}

\section{The unit interval plus one point far away}
The proof of the isometric flexibility of the Wasserstein space $\mathcal{W}_1([0,1], |\cdot|)$ was given in \cite{geher_isometric_2020} by Geh\'er, Titkos and Virosztek. They showed that, besides the trivial isometries $\mathrm{id}$ and $r_{\#}$, where $r:[0,1] \to [0,1]$ is the reflection given by $r(x)=1-x$, this Wasserstein space admits a so called \textit{flip} isometry denoted by $J$ (as well as the composition with the reflection $r_{\#}J=Jr_{\#}$). To define this flip isometry $J$, the authors used the cumulative distribution function of a measure, $F_{\mu}: [0,1] \to [0,1]$, defined as 
\begin{equation}
F_{\mu}(x)=\mu([0,x]).
\end{equation}
Then the flip isometry is defined by $J(\mu)=\nu$ such that $F_{\nu}=F^{-1}_{\mu}$, where $F^{-1}_{\mu}$ is the generalized inverse of $F_{\mu}$, as described in \cite[Subsec. 2.2.1]{villani_optimal_2009}. This flip isometry does not send Dirac masses to Dirac masses; for example, $J(\delta_{1/2})=\frac{1}{2}\delta_0 + \frac{1}{2}\delta_1$. 

The distance we will consider in this section assumes that the point $q$ is located at a constant distance $D$ from all points of the interval:
\begin{align} \label{eq:unif-dist-def}
    d_1(x,y)=\abs{x-y} \text{ for all } x,y \in[0,1], \text{ and } d_1(x,q)=D \gg 1 \text{ for all } x \in [0,1].
\end{align}
The aim of this section is to prove Theorem \ref{thm_1} where the distance is defined by \eqref{eq:unif-dist-def}. Before we give the proof, we will need a series of preparatory Lemmata. Our first step is the following characterization of the measure $\delta_q$. 
\begin{lemma}   
\label{lem:charac_1}
For a measure $\mu \in \mathcal{W}_1([0,1]\cup\{q\}, d_1),$ the following are equivalent:
\begin{enumerate}
	\item \label{item:mu-is-delta-q_case1} $\mu=\delta_q,$
    \item \label{item:big-triangle} there exist measures $\nu_1, \nu_2 \in \mathcal{W}_1([0,1]\cup\{q\}, d_1)$ such that $d_{\mathcal{W}_1}(\mu,\nu_i)=D$ for $i=1,2$ and $d_{\mathcal{W}_1}(\nu_1,\nu_2)=1.$
\end{enumerate}
\end{lemma}
\begin{proof}
The direction \eqref{item:mu-is-delta-q_case1} $\Rightarrow$ \eqref{item:big-triangle} is easy to see by taking $\nu_1:=\delta_0$ and $\nu_2:= \delta_1.$ The converse direction \eqref{item:big-triangle} $\Rightarrow$ \eqref{item:mu-is-delta-q_case1} follows from the fact that if $d_{\mathcal{W}_1}(\mu,\nu_1)=D,$ then either $\supp(\mu)=\{q\}$ and $\supp(\nu_1) \subseteq [0,1]$ or the other way around, but if $\supp(\mu) \subseteq [0,1],$ then $\supp(\nu_i)=\{q\}$ for $i=1,2$ which contradicts $d_{\mathcal{W}_1}(\nu_1,\nu_2)=1.$ So $\supp(\mu)=\{q\},$ which means that $\mu=\delta_q.$ 
\end{proof}    

From this metric characterization, we can get the important corollary which shows that any isometry of the Wasserstein space leaves the measure $\delta_q$ invariant.
\begin{corollary}
Assume that $\Phi: \mathcal{W}_1([0,1] \cup \{q\}, d_1) \to \mathcal{W}_1([0,1] \cup \{q\}, d_1)$ is an isometry. Then $\Phi(\delta_q)=\delta_q$.  
\end{corollary}
\begin{proof}
By Lemma \ref{lem:charac_1}, there exist $\nu_1, \nu_2$ such that $d_{\mathcal{W}_1}(\delta_q,\nu_i)=D$ for $i=1,2$ and $d_{\mathcal{W}_1}(\nu_1,\nu_2)=1,$ and hence $d_{\mathcal{W}_1}(\Phi(\delta_q),\Phi(\nu_i))=D$ for $i=1,2$ and $d_{\mathcal{W}_1}(\Phi(\nu_1),\Phi(\nu_2))=1$. Lemma \ref{lem:charac_1} then implies that $\Phi(\delta_q)=\delta_q.$
\end{proof}

In the next step, we separate our Wasserstein space into slices that share the same weight on the point $q$. That is, we define for $t \in [0,1]$ the $t$-slice $S_t$ as
\begin{equation}
\label{eq:t-slice-def}
S_t=\{\mu \in \mathcal{W}_1([0,1] \cup \{q\}, d_1) | \mu(\{q\})=1-t \}.
\end{equation}

We can now show that any isometry globally preserves $t$-slices $S_t$.
\begin{lemma}
\label{lem:Slices_preserved_case1}
For $t \in [0,1]$ and $\Phi: \mathcal{W}_1([0,1] \cup \{q\}, d_1) \to \mathcal{W}_1([0,1] \cup \{q\}, d_1)$ an isometry, we have $\Phi(S_t) = S_t$.
\end{lemma}
\begin{proof} We will in fact only prove that $\Phi(S_t) \subseteq S_t$. Indeed, since $\Phi^{-1}$ is also an isometry, the surjectivity of $\Phi$ and the inclusion chain $S_t=\Phi(\Phi^{-1}(S_t)) \subseteq \Phi(S_t) \subseteq S_t$ shows that $\Phi(S_t) \subseteq S_t$ for all isometries $\Phi$ implies $\Phi(S_t)=S_t$.

To check the inclusion $\Phi(S_t) \subseteq S_t$, we observe that
\begin{align} \label{eq:mass-in-terms-of-1-Wass-distance}
\mu([0,1])=\frac{1}{D} d_{\mathcal{W}_1}(\delta_q, \mu) \text{ for all } \mu \in \mathcal{W}_1([0,1]\cup\{q\}, d_1). 
\end{align}
Indeed, 
\begin{align}
        d_{\mathcal{W}_1}(\delta_q, \mu)=\int_{[0,1] \cup \{q\}} d_1(q,y) \dd \mu(y)&=\int_{[0,1]}d_1(q,y) \dd \mu(y)+\int_{\{q\}} d_1(q,y) \dd \mu(y) \nonumber \\
        &=\int_{[0,1]}D \dd \mu (y) +0=D \mu([0,1]).
    \end{align}
    Consequently, 
    \begin{align}
        \Phi(\mu)([0,1])=\frac{1}{D} d_{\mathcal{W}_1}(\delta_q,\Phi(\mu))
        =\frac{1}{D} d_{\mathcal{W}_1}(\Phi(\delta_q),\Phi(\mu)) \nonumber \\
        =\frac{1}{D} d_{\mathcal{W}_1}(\delta_q,\mu)=\mu([0,1])
    \end{align}
    for every $\mu$ and every Wasserstein isometry $\Phi,$ which implies that $\Phi(S_t) \subseteq S_t$ for all $t \in [0,1].$
\end{proof}

As a result of this Lemma, we can consider the restrictions $\Phi_{|S_t}$ to the measures of $S_t$, which will be an isometry from a $t$-slice to itself. As we will see in the next Lemma, up to a rescaling factor, each slice $S_t$ is isometric to the Wasserstein space $\mathcal{W}_1([0,1], |\cdot|)$, and thus $\Phi_{|S_t}$ acts like one of the four possible isometries $\{\mathrm{id}, r_{\#}, J, r_{\#} J\}.$ To prove Theorem \ref{thm_1} we will then rule out the cases $\Phi_{|S_t}=J$ or $\Phi_{|S_t}=r_{\#}J$. 

\begin{lemma}
For $t \in (0,1]$, the $t$-slice $S_t$ is isometric to the rescaled Wasserstein spaces $\mathcal{W}_1([0,1], t|\cdot |)$. 
\end{lemma}
\begin{proof}
The isometric relationship between the two metric spaces is given by the following map: for $\mu=(1-t) \delta_q + t \mu'$, we define the isometry $\iota(\mu)=\mu'$. It is clear that this map is surjective, we thus only need to show that $d_{\mathcal{W}_1}(\mu, \nu)=td_{\mathcal{W}_1}(\mu', \nu')$ for any $\mu, \nu \in S_t$. 

We recall that the Kantorovich-Rubinstein duality says that for a complete separable metric space $(X,d)$ and for two measures $\mu, \nu \in \mathcal{W}_1(X, d)$
$$
d_{\mathcal{W}_1}(\mu,\nu) =\sup\lers{\int_{X} f \dd \nu - \int_X f \dd \mu \, \middle| \, f: X \to \R, \, \norm{f}_{Lip} \leq 1 };
$$ see for example \cite[Prop. 2.6.6]{figalli_invitation_2021}. 
Let $t \in (0,1],$ and let $\mu, \nu \in S_t$ with the decomposition 
\begin{align*}
        \mu=(1-t) \delta_q  + t \mu' \text{ and } \nu=(1-t) \delta_q+t \nu'
\end{align*} for $\mu', \nu' \in \mathcal{W}_1([0,1],|\cdot |)$. 
Then 
\begin{align} \label{eq:scaling-of-the-1-Wasserstein-distance}
d_{\mathcal{W}_1}(\mu,\nu)&=\sup\lers{\int_{[0,1] \cup\{q\}} f \dd \nu-\int_{[0,1] \cup\{q\}} f \dd \mu \, \middle| \, f: [0,1] \cup\{q\} \to \R, \, \norm{f}_{Lip} \leq 1 } \nonumber \\
&=\sup\lers{ \int_{[0,1] \cup\{q\}} f \dd (t \nu' +(1-t) \delta_q)-\int_{[0,1] \cup\{q\}} f \dd (t \mu' +(1-t) \delta_q)\, \middle| \, \norm{f}_{Lip} \leq 1} \nonumber \\
&=t \sup\lers{ \int_{[0,1] \cup\{q\}} f \dd \nu'- \int_{[0,1] \cup\{q\}}f \dd \mu'\, \middle| \, \norm{f}_{Lip} \leq 1}= t d_{\mathcal{W}_1}(\mu',\nu').
\end{align}
\end{proof}

Let $\mu$ and $\nu \in S_t$, for $t \in (0,1]$.
Since by Lemma \ref{lem:Slices_preserved_case1} we have $\Phi(\mu), \Phi(\nu) \in S_t$, we can write $\Phi(\mu)=(1-t)\delta_q+t\Phi(\mu)'$ and $\Phi(\nu)=(1-t)\delta_q+t\Phi(\nu)'$ (with $\Phi(\mu)', \Phi(\nu)' \in \mathcal{W}_1([0,1], |\cdot|)).$ By \eqref{eq:scaling-of-the-1-Wasserstein-distance} and the isometric property of $\Phi,$ we get
    \begin{align}
        d_{\mathcal{W}_1}(\Phi(\mu)',\Phi(\nu)')=\frac{1}{t}d_{\mathcal{W}_1}(\Phi(\mu),\Phi(\nu))=\frac{1}{t}d_{\mathcal{W}_1}(\mu,\nu)=d_{\mathcal{W}_1}(\mu', \nu').
    \end{align}
    This implies that for every $t \in (0,1]$ there is a $\varphi_t \in \isom\ler{\mathcal{W}_1([0,1], d_1)}=\{\mathrm{id}, r_{\#}, J, r_{\#} J\}$ such that the restriction $\Phi_{|S_t}$ of $\Phi$ to $S_t$ is given by
    \begin{align} \label{eq:action-of-Phi-on-slices}
        \Phi((1-t) \delta_q+t\mu')=(1-t) \delta_q +t \varphi_t(\mu') \text{ for all } \mu' \in \mathcal{W}_1([0,1],|\cdot|).
\end{align}

With this observation, we can now finish the proof of Theorem \ref{thm_1}. 

\begin{proof}[Proof of Thm \ref{thm_1}]
Since the isometries of $[0,1] \cup \{q\}$ are given by the set $\{\mathrm{id}, r\}$ (where $r$ is extended to $([0,1] \cup \{q\}, d_1)$ by setting $r(q)=q$),
 we want to show that either $\varphi_t = \mathrm{id}$ or $\varphi_t=r_{\#}$ for any $t \in [0,1]$. We first notice that since $S_0=\{\delta_q\}$, this holds trivially for $t=0$. 

For $t \in (0,1]$, assume by contradiction that $\varphi_t \in \{J, r_{\#}J \}$, and consider the measures $\mu \in S_t$ and $\nu \in S_{2t/3}$ given by 
$$
\mu= (1-t)\delta_q + \frac{t}{2}(\delta_0 + \delta_1) \quad \text{and} \quad \nu=\left(1-\frac{2t}{3}\right)\delta_q + \frac{2t}{3} \delta_0. 
$$
The optimal transport plan from $\mu$ to $\nu$ is to send a mass of magnitude $t/6$ from $1$ to $0,$ and a mass of magnitude $t/3$ from $1$ to $q.$ Therefore, $d_{\mathcal{W}_1}(\mu,\nu)=t(D/3+1/6).$ 

By \eqref{eq:action-of-Phi-on-slices} and our assumption on $\varphi_t$, $\Phi(\mu)=(1-t) \delta_q + t\delta_{1/2}.$ Furthermore, \eqref{eq:action-of-Phi-on-slices} shows that $\Phi(\nu)=(1-\frac{2t}{3})\delta_q + \frac{2t}{3} \delta_1$ (if $\varphi_{2t/3} \in \{r_{\#},J\}$) or $\Phi(\nu)=(1-\frac{2t}{3})\delta_q + \frac{2t}{3} \delta_0$ (if $\varphi_{2t/3} \in \{\mathrm{id},r_{\#}J\}$).
In any case, 
    \begin{align}
    d_{\mathcal{W}_1}(\Phi(\mu),\Phi(\nu))=d_{\mathcal{W}_1}((1-t) \delta_q+ \delta_{1/2},\Phi(\nu))=\frac{t}{3}D + \frac{2t}{3}\cdot \frac{1}{2}=t(D/3+1/3),    
    \end{align}
which means that $d_{\mathcal{W}_1}(\Phi(\mu),\Phi(\nu))>d_{\mathcal{W}_1}(\mu, \nu).$ This contradiction to the isometric property of $\Phi$ implies that $\varphi_t \in \{\mathrm{id}, r_{\#}\}$ for any $t \in (0,1]$.

It remains to show that $\varphi_t=\varphi_{t'}$ for any pair $t, t' \in (0,1]$. To see this, we calculate the following distances: if we consider the measures $\mu_t=(1-t) \delta_q + t\delta_0$ and $\nu_t=(1-t) \delta_q + t\delta_1$, we have 
$$
d_{\mathcal{W}_1}(\mu_t,\mu_{t'})=D|t-t'|
$$
and
$$
d_{\mathcal{W}_1}(\mu_t, \nu_{t'})=(D+1)|t-t'|.
$$ 
Thus, if there exists $t,t' \in (0,1]$ such that $\varphi_t=\mathrm{id}$ but $\varphi_{t'}=r_{\#}$, we have 
$d_{\mathcal{W}_1}(\mu_t, \mu_{t'})=D|t-t'|,$ but $d_{\mathcal{W}_1}(\Phi(\mu_t), \Phi(\mu_{t'}))=d_{\mathcal{W}_1}(\mu_t, \nu_{t'}) =(D+1)|t-t'|$, contradicting the distance preserving property of isometries. Thus either $\Phi=\mathrm{id}$ or $\Phi=r_{\#}$ on the whole Wasserstein space, proving the theorem. 
\end{proof}

\begin{remark}
We can use similar techniques to prove the rigidity of the Wasserstein space $\mathcal{W}_1([0,1] \cup \{q\}, d)$ for different metrics $d$, for example when $q$ is a point to the left of the interval, such that the distance between $q$ and $x \in [0,1]$ is given by $d(q, x)=x + \varepsilon$, for some $\varepsilon>0$. 
\end{remark}

%%%%%%%%%%%%%%%%%%%%%%%%%%%%%%%%%%%%%%%%%%%%%%%%%%%%%%%%%%%%%%%%%%%%%%
%%%%%%%%%%%%%%%%%%%%%%%%%%%%%%%%%%%%%%%%%%%%%%%%%%%%%%%%%%%%%%%%%%%%%%
\section{The $2$-Wasserstein space over the Euclidean plane plus one point}

In \cite{kloeckner_geometric_2010}, Kloeckner showed that the Wasserstein space $\mathcal{W}_2(\Rn, d_E)$ admits flexible isometries for $n \geq 2$. In fact, he gave an explicit description of the possible isometries:  for a linear isometry $\varphi$ of $(\Rn, d_E)$, the map 
$$\Phi_{\varphi}(\mu)=\varphi_{\#}(\mu-g)+g$$
 is an isometry of the Wasserstein space $\mathcal{W}_2(\Rn, d_E)$, where $g$ denotes both the center of mass of the measure $\mu$ as well as the corresponding translation. He also showed that any isometry of the Wasserstein space is a composition of a trivial isometry and an isometry of the form $\Phi_{\varphi}$. 
 
In the following, we will focus on the case $n=2$.  
 
Morally, the isometries $\Phi_{\varphi}$ can be described as following: given a linear isometry $\varphi$ of $\R^2$ (which is a rotation of a certain angle or a reflection with respect to an axis), the Wasserstein isometry $\Phi_{\varphi}$ rotates or reflects the measure $\mu$ around its own center of mass. We can observe that Dirac masses are preserved by these isometries. 

In this section we consider the distance $d_2$, defined by 
\begin{align}
\label{eq:euclid_3_distance}
    d_2(q,x)= \sqrt{1+\norm{x}^2} \text{ and } d_2(x,x')=\norm{x-x'} \text{ for all } x, x' \in \R^2.
\end{align}
This metric can be represented by horizontally embedding $(\R^2, d_E)$ into $(\R^3, d_E)$, setting $q=(0,0,1)$, and taking the $3$-dimensional Euclidean distance as its metric. 

In this section we prove Theorem \ref{thm_2} where the distance is defined by \eqref{eq:euclid_3_distance}. The proof is based on the following Lemmata, the first being the characterization of $t-$slices
%%%%%%%%%%%%%%%%%%%%%%%%%%%%%%%%%%%%%%%%%%%%%%%%%%%%%%%%%%%%%%%%%%%%%%
\begin{equation}
\label{eq:t-slice-def2}
S_t=\{\mu \in \mathcal{W}_1(\R^2 \cup \{q\}, d_2) | \mu(\{q\})=1-t \}
\end{equation}
by geodesics in the Wasserstein space $\mathcal{W}_2(\R^2 \cup \{q\}, d_2)$.
\begin{lemma}
\label{lem:slice_geod}
Consider $t \in [0,1]$ and set $\nu_t= (1-t)\delta_{q} + t \delta_0$. Then, for any measure $\mu \in \mathcal{W}_2(\R^2 \cup \{q\}, d_2)$, the following are equivalent:
\begin{enumerate}
\item The measure $\mu$ is in the slice $S_t$.
\item There exists a geodesic between $\mu$ and $\nu_t$. 
\end{enumerate}
\end{lemma}
\begin{proof}
We start by assuming that $\mu \in S_t$, and write $\mu= (1-t)\delta_{q} + t \mu'$ for some $\mu' \in \mathcal{W}_2(\R^2, d_E)$. Consider the map $T:\R^2 \cup \{q\} \to \R^2 \cup \{q\}$ given by $T(q)=q$ and $T(x)=0$ for any $x \in \R^2$. Then $\pi_0=(\mathrm{id} \times T)_{\#}$ is a transport plan (i.e. a coupling) between $\mu$ and $\nu_t$ that sends any mass supported on $\R^2$ to $0$, and does not move the mass at $q$. It has the transport cost 
$$
\int_{\R^2 \cup \{q\} \times \R^2 \cup \{q\}}d_2^2(x,y)  d\pi_0(x,y) = t\int_{\R^2}d_2^2(x,0)d\mu'(x).
$$

To prove the optimality of $\pi_0$ we use the Kantorovich duality. By \cite[Theorem 5.10]{villani_optimal_2009}, the transport plan $\pi_0$ is optimal if and only if there exist $\psi: \R^2 \cup \{q\} \to \R \cup \{+\infty\}$ and $\varphi: \R^2 \cup \{q\} \to \R \cup \{-\infty\}$ such that $\varphi(z)-\psi(x) \leq d_2^2(x,z)$ for all $(x,z) \in \R^2 \cup \{q\} \times \R^2 \cup \{q\},$ with equality $\pi_0$-almost surely. Set
\begin{align} \label{eq:Kantorovich-potential-def}
    \psi(x)=\begin{cases}
        -\norm{x}^2 & \text{if } x \in \R^2 \\
        0 & \text{if } x=q
    \end{cases}
    \text{ and }
    \varphi(z)=\begin{cases}
        0 & \text{if } z \in \{q,0\} \\
        -\infty & \text{otherwise.}
    \end{cases}
\end{align}
The inequality $\varphi(z)-\psi(x) \leq d_2^2(x,z)$ clearly holds when $z \notin \{0,q\}.$ For $z=q,$ we have $0-(-\norm{x}^2)=\norm{x}^2\leq d_2^2(x, q)$ for all $x \in \R^2,$ and $0-0 \leq d_2^2(q,q).$ The $z=0$ case is very similar: $0-(-\norm{x}^2)=\norm{x}^2\leq d_2^2(x, 0)$ for all $x \in \R^2,$ and $0-0 \leq d_2^2(q,0).$ Moreover,
\begin{align} \label{dual-equals-primal}
    \int_{\R^2 \cup \{q\}}\varphi(z) \dd \nu_t(z)- \int_{\R^2 \cup \{q\}} \psi(x) \dd \mu(x)=0+\int_{\R^2} \norm{x}^2 \dd \mu(x),
\end{align}
which is precisely the cost of the transport plan $\pi_0.$ Consequently, $\varphi(z)-\psi(x)= d_2^2(x,z)$ $\pi_0$-almost surely, and hence $\pi_0$ is optimal. Thus we can construct the geodesic
$\gamma: [0,1] \to \mathcal{W}_2(\R^2 \cup \{q\}, d_2)$ by setting
$\gamma(s)=p_{s\#}\pi_0$, where $p_s(x,y)=(1-s)x + sy$. Since $\pi_0(\{q\} \times \R^2)=0$, this geodesic is indeed well-defined.  

To see the converse direction, we adapt an argument given by Maas in \cite{maas_gradient_2011}. Assume that $\gamma: [0,1] \to \mathcal{W}_2(\R^2 \cup \{q\}, d_2)$ is a (constant speed) geodesic connecting $\mu$ and $\nu_t$, that is, 
\begin{align}
d_{\mathcal{W}_2}(\gamma(s),\gamma(s'))=C \abs{s-s'} \text{ for all } s, s' \in [0,1], \, \gamma(0)=\mu, \, \gamma(1)=\nu_t.
\end{align}
Note that one can turn every geodesic into a constant speed geodesic by re-parametrization.
Let $m: [0,1]\to [0,1]$ be defined by the relation $\gamma(s)\in S_{m(s)}$ for all $s \in [0,1],$ that is, $m(s)$ is the mass that $\gamma(s)$ assigns to $\R^2.$ So $\gamma(s)$ is in $S_{m(s)}$ and $\gamma(s')$ is in $S_{m(s')}$ for all $s,s' \in [0,1],$ and hence when transporting from $\gamma(s)$ to $\gamma(s'),$ at least a mass of magnitude $\abs{m(s)-m(s')}$ must be transported for a distance at least $\mathrm{dist}(y_0,\R^2)=1.$ Therefore,
\begin{align}
    \abs{m(s)-m(s')} \leq d_{\mathcal{W}_2}^2(\gamma(s),\gamma(s'))=C^2\abs{s-s'}^2,
\end{align}
which means that $s \mapsto m(s)$ is $2$-H\"older continuous on $[0,1],$ and hence constant. Clearly, $m(1)=t,$ so $m(0)=t$ as well, that is, $\mu \in S_t.$
\end{proof}

A consequence of the above Lemma is that if $\Phi: \mathcal{W}_2(\R^2 \cup \{q\}, d_2) \to \mathcal{W}_2(\R^2 \cup \{q\}, d_2)$ is an isometry, then $\Phi$ preserves $t-$Slices, i.e. $\Phi(S_t)= S_t$ for any $t\in [0,1]$. To be more precise, we have the following result:
\begin{lemma}
\label{lem:slices_preserved_case45}
Consider an isometry $\Phi: \mathcal{W}_2(\R^2 \cup \{q\}, d_2) \to \mathcal{W}_2(\R^2 \cup \{q\}, d_2)$ of the Wasserstein space. Then the following results hold:
\begin{enumerate}
\item The Dirac mass at $q$ is preserved, i.e. $\Phi(\delta_{q})=\delta_{q}$.
\item If $\nu_t=(1-t)\delta_{q} + t \delta_0$, then $\Phi(\nu_t)=\nu_t$.
\item Slices are preserved, that is $\Phi(S_t)=S_t$.
\end{enumerate}
\end{lemma} 
\begin{proof}
The first result follows from a simple observation: since $\delta_{q}=\nu_0$ and since $S_0=\{\delta_q\}$ is a slice with a single element (in fact it is the only slice with a single element), by Lemma \ref{lem:slice_geod} there exists no measure $\mu \in \mathcal{W}_2(\R^2 \cup \{q\}, d_2)$ such that there is a geodesic between $\delta_{q}$ and $\mu$ (except the trivial geodesic when $\mu= \delta_q$). Thus there exists no measure $\mu' \in \mathcal{W}_2(\R^2 \cup \{q\}, d_2)$ such that there is a non-trivial geodesic between $\Phi(\delta_{q})$ and $\mu'$. Again by Lemma \ref{lem:slice_geod}, we then have that $\Phi(\delta_{q})\in S_0$, in other words $\Phi(\delta_{q})=\delta_{q}$, proving the first claim. 

To show the second claim, consider for $t \in (0,1]$ the measure $\nu_t$, and assume that $\Phi(\nu_t)\in S_{t'}$ for some $t' \in [0,1]$. By Lemma \ref{lem:slice_geod} there exists a geodesic between $\Phi(\nu_t)$ and $\nu_{t'}$. Since $\Phi$ is an isometry, there then also exists a geodesic between $\nu_t$ and $\Phi^{-1}(\nu_{t'})$, and thus $\Phi^{-1}(\nu_{t'}) \in S_t$. 

From the definition of the distance, it is easy to see that the measure $\nu_t$ is the measure in the slice $S_t$ closest to $\delta_{q}$, i.e. 
\begin{equation}
\label{eq:slice_minimizer}
\nu_t = \argmin_{\mu \in S_t} d_{\mathcal{W}_2}(\mu, \delta_{q}),
\end{equation}
and that this argmin is unique; moreover, $t \to d_{\mathcal{W}_2}(\nu_t, \delta_{q})=\sqrt{t}$ is a strictly increasing function. In particular since 
$$d_{\mathcal{W}_2}(\nu_t, \delta_{q})=d_{\mathcal{W}_2}(\Phi(\nu_t), \Phi(\delta_{q}))=d_{\mathcal{W}_2}(\Phi(\nu_t), \delta_{q}),$$
$\Phi(\nu_t)\in S_{t'}$ implies that 
$$d_{\mathcal{W}_2}(\Phi(\nu_t), \delta_{q}) \geq d_{W_2}(\nu_{t'}, \delta_{q})$$
and that $t' \geq t$. Since $\Phi^{-1}$ is also an isometry and $\Phi^{-1}(\nu_{t'}) \in S_t$, a similar argument shows that $t \geq t'$, implying $t=t'$. Thus, since $\nu_t$ is the unique minimizer of \eqref{eq:slice_minimizer}, the equality $d_{\mathcal{W}_2}(\Phi(\nu_t), \delta_{q}) = d_{\mathcal{W}_2}(\nu_t, \delta_{q})$ implies the second claim. 

The third claim now follows easily: For any $\mu \in S_t$, there exists a geodesic between $\mu$ and $\nu_t$. By the distance-preserving properties of the isometry, there also exists a geodesic between $\Phi(\mu)$ and $\Phi(\nu_t)=\nu_t$, and by Lemma \ref{lem:slice_geod} it follows that $\Phi(\mu) \in S_t$, proving the claim. 
\end{proof}

We now define the projection operator onto the slice $S_1$:
$$
P_{\#}\mu=\argmin_{\nu \in S_1} d_{\mathcal{W}_2}(\mu, \nu).
$$ 
It is easy to see that, if $\mu=(1-t)\delta_{q} + t \mu'$ (with $\mu' \in \mathcal{W}_2(\R^2, d_E)$), then the argmin is unique and given by $P_{\#}\mu=(1-t)\delta_0 + t\mu'$. 

An important property of the projection operator is that it commutes with isometries of the Wasserstein space. 
\begin{lemma}
If $\Phi: \mathcal{W}_2(\R^2 \cup \{q\}, d_2) \to \mathcal{W}_2(\R^2 \cup \{q\}, d_2)$ is an isometry of the Wasserstein space, we have 
\begin{equation}
    \label{eq:projection_switch}
    P_{\#}(\Phi(\mu))=\Phi(P_{\#}(\mu)).
\end{equation}  
\end{lemma}
\begin{proof}
This property follows simply from the facts that $\Phi(S_1) = S_1$ and that isometries are surjective and distance-preserving:
$$
P_{\#}(\Phi(\mu))=\argmin_{\nu \in S_1} d_{\mathcal{W}_2}(\Phi(\mu), \nu)=\argmin_{\nu \in S_1} d_{\mathcal{W}_2}(\mu, \Phi^{-1}(\nu))= \Phi(P_{\#}(\mu)).
$$
\end{proof}

\begin{proof}[Proof of Theorem \ref{thm_2}]
Using the previous results, we now study how isometries act on the $1$-slice $S_1$. Since $S_1$ is isometric to $\mathcal{W}_2(\R^2, d_E)$, we can think of $\Phi_{|S_1}$ as an isometry of $\mathcal{W}_2(\R^2, d_E)$. Thus there exists an isometry $\varphi$ of $(\R^2 \cup \{q\}, d_2)$ such that $\Phi \circ \varphi_{\#}$ preserves Dirac masses, $(\Phi \circ \varphi_{\#})(\delta_x)=\delta_x$ for any $x \in \R^2$. Then, up to composing with a trivial isometry we can assume that $\Phi(\delta_x)=\delta_x$. 

Thus, by the result of Kloeckner \cite{kloeckner_geometric_2010}, $\Phi_{|S_1}$ acts by rotating or reflecting measures around their center of mass. We consider the measure $\mu'=\frac{1}{3}\delta_{(0,0)} + \frac{2}{3}\delta_{(1,0)}$, then the center of mass is $x_0=(\frac{2}{3}, 0)$ and the image $\Phi(\mu')$ is given by 
$$
\Phi(\mu')=\frac{1}{3}\delta_{x_0 + 2/3(-\cos(\vartheta), -\sin(\vartheta))} + \frac{2}{3}\delta_{x_0 + 1/3 (\cos(\vartheta), \sin(\vartheta))}
$$
for some angle $\vartheta \in [0, 2\pi)$. In particular, unless $\vartheta =0$ (which is the case when $\Phi_{|S_1}$ is the identity), $\Phi(\mu')$ is not supported on the point $0 \in \R^2$. 

Now consider the measure $\mu\in S_{2/3}$ given by $\mu=\frac{1}{3}\delta_{q} + \frac{2}{3}\delta_{(1,0)}$. By Lemma \ref{lem:slices_preserved_case45} we have that $\Phi(\mu) \in S_{2/3}$. Using Lemma \ref{eq:projection_switch}, since $P_{\#}\mu=\frac{1}{3}\delta_{0} + \frac{2}{3}\delta_{(1,0)}=\mu'$, we also have that $\Phi(\mu')=P_{\#}(\Phi(\mu))$. Since $\Phi(\mu) \in S_{2/3}$, $P_{\#}(\Phi(\mu))(\{0\}) \geq 1/3 >0$, but as the above equation shows, $\Phi(\mu')(\{0\})=0$, unless $\vartheta=0$, for which $\Phi$ is the identity on $S_1$. Thus, up to composing with a trivial isometry, we have that $\Phi_{|S_1}$ is the identity. 

To finish the proof, consider an isometry $\Phi$ of $\mathcal{W}_2(\R^2 \cup \{q\}, d_2)$ such that $\Phi_{|S_1}$ is the identity, and consider $\mu \in S_t$ for $t \in [0,1]$. Restricted to the slice $S_t$, the projection operator becomes injective. Indeed, for $\nu_1, \nu_2 \in S_t$, we have 
$P_{\#}\nu_1= (1-t) \delta_0 + t \nu_1'$ and $P_{\#}\nu_2= (1-t) \delta_0 + t \nu_2'$; thus, if $P_{\#}\nu_1=P_{\#}\nu_1$, we have $\nu_1'=\nu_2'$, and thus $\nu_1=\nu_2$. 

Since $\Phi_{|S_1}$ is the identity and $P_{\#}\mu\in S_1$, equation \eqref{eq:projection_switch} becomes   
\begin{equation}
\label{eq:counter}
P_{\#}(\mu)=P_{\#}(\Phi(\mu)).
\end{equation}

As by Lemma \ref{lem:slices_preserved_case45} $\Phi(\mu) \in S_t$ if $\mu \in S_t$, the injectivity of $P_{\#}$ over measures of $S_t$ thus implies that $\mu=\Phi(\mu)$. Therefore the space $\mathcal{W}_2(\R^2 \cup \{q\}, d_2)$ is isometrically rigid. 
\end{proof}

\paragraph*{{\bf Acknowledgments}} We thank Tam\'as Titkos for helpful discussions on the topic.

%%% Bibliography

\bibliographystyle{abbrv}
\bibliography{Rigid_plus_a_point_paper.bib}

\end{document}